\documentclass[10pt]{amsart}
\usepackage{amsmath}
\usepackage{amssymb}
\usepackage{indentfirst}
\usepackage[latin1]{inputenc}
\usepackage{geometry}
\usepackage{amsfonts}
\usepackage{enumitem}
\usepackage{cancel}
\usepackage[bookmarks=false,colorlinks=true,linkcolor=black,citecolor=black,filecolor=black,urlcolor=black]{hyperref}
\DeclareSymbolFont{largesymbols}{OMX}{yhex}{m}{n}
\DeclareMathAccent{\widehat}{\mathord}{largesymbols}{"62}

\newcommand{\Q}{{\mathbb Q}}
\newcommand{\Z}{{\mathbb Z}}

\newcommand{\N}{{\mathbb N}}
\newcommand{\R}{{\mathbb R}}

\newcommand{\U}{\mathcal{U}}
\renewcommand{\O}{\mathcal{O}}
\newcommand{\NN}{\mathcal{N}}
\newenvironment{frcseries}{\fontfamily{frc}\selectfont}{}
\newcommand{\textfrc}[1]{{\frcseries#1}}
\newcommand{\cc}[1]{\textfrc{c}^{#1}}
\newcommand{\Gal}{\textnormal{Gal}}

\newcommand{\Cen}{\textnormal{Cen}}
\newcommand{\Aut}{\textnormal{Aut}}

\newcommand{\GL}{\textnormal{GL}}
\newcommand{\SL}{\textnormal{SL}}

\newcommand{\lcm}{\textnormal{lcm}}

\newcommand{\Bass}[3]{u_{#1,#2}(#3)}
\newcommand{\suma}[1]{\widehat{#1}}

\newcommand{\inv}{{^{-1}}}
\newcommand{\GEN}[1]{\left\langle #1 \right\rangle}

\newtheorem{theorem}{Theorem}[section]
\newtheorem{lemma}[theorem]{Lemma}
\newtheorem{corollary}[theorem]{Corollary}

\theoremstyle{definition}

\newtheorem{remark}[theorem]{Remark}

\begin{document}

\title[Central units]{Central units of integral group rings}

\author{Eric Jespers}
\address{Department of Mathematics, Vrije Universiteit Brussel,
Pleinlaan 2, 1050 Brussels, Belgium}
\email{efjesper@vub.ac.be}

\author{Gabriela Olteanu} 
\address{Department of Statistics-Forecasts-Mathematics, Babe\c s-Bolyai University,
Str. T. Mihali 58-60, 400591 Cluj-Napoca, Romania}
\email{gabriela.olteanu@econ.ubbcluj.ro}

\author{\'{A}ngel del R\'io}
\address{Departamento de Matem\'{a}ticas, Universidad de Murcia,  30100 Murcia, Spain}
\email{adelrio@um.es}

\author{Inneke Van Gelder}
\address{Department of Mathematics, Vrije Universiteit Brussel,
Pleinlaan 2, 1050 Brussels, Belgium}
\email{ivgelder@vub.ac.be}



\thanks{The research is partially supported by Ministerio de Ciencia y Tecnolog\'{\i}a of Spain and Fundaci\'{o}n
S\'{e}neca of Murcia, the Research Foundation Flanders (FWO - Vlaanderen), Onderzoeksraad Vrije Universiteit Brussel and by
the grant PN-II-RU-TE-2009-1 project ID\_303.}

\date{\today}

\begin{abstract}
We give an explicit description for a basis of  a subgroup of finite index in the group of central units of the integral group ring $\Z G$ of a finite
abelian-by-supersolvable  group  such that every cyclic subgroup of order not a divisor of 4 or 6 is subnormal in $G$. The basis elements turn out to
be a natural product of conjugates of Bass units. This extends and generalizes a result of Jespers, Parmenter and Sehgal showing that the Bass units
generate a subgroup of finite index in the  center $\mathcal{Z} (\U (\Z G))$ of the unit group $\U (\Z G)$ in case $G$ is a finite nilpotent group.
Next, we give a new construction of units that generate a subgroup of finite index in $\mathcal{Z}(\U(\Z G))$ for all finite strongly monomial groups
$G$. We call these units generalized Bass units. Finally, we show that the commutator group $\U(\Z G)/\U(\Z G)'$ and $\mathcal{Z}(\U(\Z G))$ have the
same rank if $G$ is a finite group such that  $\Q G$ has no epimorphic image which is either a non-commutative division algebra other than a totally
definite quaternion algebra, or a two-by-two matrix algebra over a division algebra with center either the rationals or a quadratic imaginary
extension of $\Q$. This allows us to prove that in this case the natural images of the Bass units of $\Z G$ generate a subgroup of finite index in
$\U(\Z G)/\U(\Z G)'$.
\end{abstract}

\maketitle

\section{Introduction}

Let $\mathcal{Z}(\U(\Z G))$ denote the group of central units in the integral group ring $\Z G$, for $G$ a finite group. Then $\mathcal{Z}(\U(\Z G))$
is equal to $\pm \mathcal{Z}(G) \times T$, where $T$ is a finitely generated free abelian subgroup of $\mathcal{Z}(\U(\Z G))$ \cite[Corollary
7.3.3]{2002SehgalMilies}. 

Bass proved that if $G$ is a finite cyclic group, then the so-called Bass units (also known as Bass cyclic units) generate a
subgroup of finite index in $\U(\Z G)$ \cite{bass1966}. Moreover, he described an independent set of generators using the
Bass Independence Lemma. In these investigations the cyclotomic units show up and therefore the Bass units are a natural
choice. Next, Bass and Milnor proved this result for finite abelian groups reducing to cyclic groups and using $K$-theory.
However, they did not describe an independent set of generators. Recently, these results were proved avoiding $K$-theory and
the Bass Independence Lemma, and an independent set of generators is described \cite{2012JdRVG}.

In this paper we construct a basis for a subgroup of finite index in the center of $\U (\Z G)$ provided the finite group $G$ is  abelian-by-supersolvable and has the  property that
every cyclic subgroup of order not a divisor of $4$ or $6$ is subnormal in $G$. The basis elements are constructed as a (natural) product of conjugates of Bass cyclic units. For finite nilpotent groups $G$, constructions of central units of this type
have earlier been considered (using K-theory) by   Jespers, Parmenter, Sehgal \cite{JesParSeh1996} in the context of finding finitely many generators for a subgroup of finite index in the center
of $\U (\Z G)$. Ferraz and Sim\'on in \cite{FerrazSimon2008} constructed a basis for the center of $\U (\Z G)$ in case $G$ is a metacyclic group of
order $pq$, with $p$ and $q$ two distinct odd primes.

Next, for  arbitrary finite strongly monomial groups $G$,  we construct   generalized Bass units and show that the group generated by these units contains a subgroup of finite index in $\mathcal{Z}(\mathcal{U}(\Z G))$. This generalizes a result of Jespers and Parmenter \cite{JesPar2011} on metabelian groups.

For many finite groups $G$, it was proved that the group $B_G$ generated by the Bass units and the bicyclic units (of one type) has finite index in $\U(\Z G)$ (see for example \cite{RS1989,RS1991}). In \cite{RS1989} it was proved that if $G$ is a finite nilpotent group such that $\Q G$ does not have in its Wedderburn decomposition certain types of simple algebras, called exceptional components, then $B_G$ has finite index in $\U(\Z G)$. Furthermore, Jespers and Leal have extended these results to a much larger class of groups \cite{JespersLeal1993}. It is proved that $B_{G}$ is of finite index in $\U (\Z G)$ if $G$ is a finite group such that $\Q G$ has no exceptional components and $G$ has no non-abelian homomorphic image which is fixed point free. In this paper we also show that for such groups the group generated by the bicyclic units is of finite index in the commutator group $\U (\Z G)'$. Furthermore, $\U(\Z G)/\U(\Z G)'$ and $\mathcal{Z}(\U(\Z G))$ (and $K_1(\Z G)$) have the same rank. This allows us to prove that in this case the natural images of the Bass units of $\Z G$ generate a subgroup of finite index in $\U(\Z G)/\U(\Z G)'$.

\section{Preliminaries}

We first recall the definition of Bass units, a classical construction of units in integral group rings.

Let $G$ be a finite group, $g$ an element of $G$ of order $n$ and $k$ and $m$ positive integers such that $k^m\equiv 1 \mod n$. Then the Bass unit 
based on $g$ with parameters $k$ and $m$ is
$$\Bass{k}{m}{g}=(1+g+\dots + g^{k-1})^{m}+\frac{1-k^m}{n}(1+g+\dots+g^{n-1}).$$ The elements of this form were introduced in \cite{bass1966}, 
they are units in the integral group ring $\Z G$ and satisfy the following equalities (\cite[Lemma 3.1]{2006GoncalvesPassman}):
\begin{eqnarray}
 \label{Basseq1} \Bass{k}{m}{g}&=&\Bass{k_1}{m}{g}, \mbox{ if } k\equiv k_1 \mod n, 
\end{eqnarray} and hence we allow negative integers $k$ with the obvious meaning.
\begin{eqnarray}
 \label{Basseq2} \Bass{k}{m}{g}\Bass{k}{m_1}{g}&=&\Bass{k}{m+m_1}{g},\\
 \label{Basseq3} \Bass{k}{m}{g}\Bass{k_1}{m}{g^k}&=&\Bass{kk_1}{m}{g},\\
 \label{Basseq4} \Bass{1}{m}{g}&=&1 \text{ and } \\
 \label{Basseq5} \Bass{-1}{m}{g} &=& (-g)^{-m},
\end{eqnarray}
for $g\in G$, $n=|g|$ and $k^m\equiv k_1^m \equiv k^{m_1} \equiv 1 \mod n$. By (\ref{Basseq2}) we have
\begin{equation}\label{Basseq6}
\Bass{k}{m}{g}^{i} = \Bass{k}{mi}{g}
\end{equation}
for a non-negative integer $i$ and from (\ref{Basseq1}), (\ref{Basseq3}) and (\ref{Basseq4}), we have 
\begin{eqnarray}
 \label{Basseq7} \Bass{k}{m}{g}\inv = \Bass{k_1}{m}{g^k},
\end{eqnarray}
if $kk_1\equiv 1 \mod n$. Thus an integral power of a Bass unit is a Bass unit.
Furthermore, from
(\ref{Basseq1}), (\ref{Basseq3}) and (\ref{Basseq5}) we deduce
    \begin{equation}\label{Basseq8}
    \Bass{n-k}{m}{g} = \Bass{k(n-1)}{m}{g} = \Bass{k}{m}{g} \Bass{n-1}{m}{g^k} = \Bass{k}{m}{g}g^{-km}
    \end{equation}
provided $(-1)^m\equiv 1 \mod n$.

Let $N$ be a normal subgroup of $G$. Using equations (\ref{Basseq1}) and (\ref{Basseq6}) together with the Chinese Remainder Theorem, it is easy 
to verify that a power of a Bass unit in $\Z(G/N)$ is the natural image of a Bass unit in $\Z G$.

If $R$ is an associative ring and $G$ is a group then $R*^{\alpha}_{\tau} G$ denotes the crossed product with action $\alpha:G\rightarrow \Aut(R)$ 
and twisting $\tau:G\times G \rightarrow \U(R)$ \cite{Passman1989}, i.e. $R*^{\alpha}_{\tau} G$ is the associative ring $\bigoplus_{g\in G} R u_g$
with multiplication given by the following rules: $u_g a = \alpha_g(a) u_g$ and $u_g u_h = \tau(g,h) u_{gh}$, for $a\in R$ and $g,h\in G$. Recall that
a classical crossed product is a crossed product $L*^{\alpha}_{\tau} G$, where $L/F$ is a finite Galois extension, $G = \Gal(L/F)$ is the Galois group
of $L/F$ and $\alpha$ is the natural action of $G$ on $L$. The classical crossed product $L *^{\alpha}_{\tau} G$ is denoted by $(L/F,\tau)$
\cite{Reiner1975}.

Our approach is making use of the description of the Wedderburn decomposition of the rational group algebra $\Q G$. We
shortly recall the character-free method of Olivieri, del R\'io and Sim\'on \cite{Olivieri2004} for a certain class of
groups, called strongly monomial groups.

Throughout, $G$ will be a finite group. If $H$ is a subgroup of $G$ then $N_G(H)$ denotes the normalizer of $H$ in $G$. We use the exponential
notation for conjugation: $a^b = b\inv a b$. For each $\alpha \in \Q G$, $C_G(\alpha)$ denotes the centralizer of $\alpha$ in $G$.

For a subgroup $H$ of $G$, let $\suma{H}=\frac{1}{|H|}\sum_{h\in H} h$. Clearly, $\suma{H}$ is an idempotent of $\Q G$ which is central if and only if
$H$ is normal in $G$. If $K\lhd H\leq G$ then let $$\varepsilon(H,K)=\prod_{M/K\in\mathcal{M}(H/K)}
(\suma{K}-\suma{M})=\suma{K}\prod_{M/K\in\mathcal{M}(H/K)} (1-\suma{M}),$$ where $\mathcal{M}(H/K)$ denotes the set of all minimal normal subgroups of
$H/K$. We extend this notation by setting $\varepsilon(H,H)=\suma{H}$. Clearly $\varepsilon(H,K)$ is an idempotent of the group algebra $\Q G$. Let
$e(G,H,K)$ be the sum of the distinct $G$-conjugates of $\varepsilon(H,K)$, that is, if $T$ is a right transversal of $C_G(\varepsilon(H,K))$ in $G$,
then
    $$e(G,H,K)=\sum_{t\in T}\varepsilon(H,K)^t.$$
Clearly, $e(G,H,K)$ is a central element of $\Q G$ and if the $G$-conjugates of $\varepsilon(H,K)$ are orthogonal, then $e(G,H,K)$ is a central 
idempotent of $\Q G$.

A strong Shoda pair of $G$ is a pair $(H,K)$ of subgroups of $G$ with the properties that $K\leq H\unlhd N_G(K)$, $H/K$ is cyclic and a maximal 
abelian subgroup of $N_G(K)/K$ and the different conjugates of $\varepsilon(H,K)$ are orthogonal. In this case $C_G(\varepsilon(H,K))=N_G(K)$
\cite{Olivieri2004}. 

Let $\theta$ be a linear character of a subgroup $H$ of $G$ with kernel $K$. Then the induced character $\theta^G$
is irreducible if and only if $H/K$ is cyclic and $[H,g]\cap H \not\subseteq K$ for every $g\in G\setminus H$
\cite{Shoda1933}. A pair $(H,K)$ of subgroups $K\unlhd H$ satisfying these conditions is called a Shoda pair. Therefore an
irreducible character $\chi$ of $G$ is monomial if and only if there is a Shoda pair $(H,K)$ of $G$ such that $\chi=\theta^G$
for a linear character $\theta$ of $H$ with kernel $K$. If $\chi=\theta^G$ for $\theta$ as above with $(H,K)$ a strong Shoda
pair of $G$ then we say that the character $\chi$ is strongly monomial. The group $G$ is strongly monomial if every
irreducible character of $G$  is strongly monomial.

For finite strongly monomial groups, including abelian-by-supersolvable groups, all primitive central idempotents are
realized by strong Shoda pairs, i.e. they are of the form $e(G,H,K)$, with $(H,K)$ a strong Shoda of $G$. However, different
strong Shoda pairs can contribute to the same primitive central idempotent. Indeed, let $(H_1,K_1)$ and $(H_2,K_2)$ be two
strong Shoda pairs of a finite group $G$. Then $e(G,H_1,K_1)=e(G,H_2,K_2)$ if and only if there is a $g\in G$ such that
$H_1^g\cap K_2=K_1^g\cap H_2$ \cite{Olivieri2006}. In that case we say that $(H_1,K_1)$ and $(H_2,K_2)$ are equivalent as
strong Shoda pairs of $G$. In particular, to calculate the primitive central idempotents of $G$ if $G$ is strongly monomial,
it is enough to consider only one strong Shoda pair in each equivalence class. We express this by saying that we take a
complete and non-redundant set of strong Shoda pairs.

The structure of the simple component $\Q Ge(G,H,K)$ is given in the following Theorem.

\begin{theorem}\label{SSP}\cite[Proposition 3.4]{Olivieri2004}
Let $(H,K)$ be a strong Shoda pair and let $k=[H:K]$, $N=N_G(K)$, $n=[G:N]$, $yK$ a generator of $H/K$ and $\phi:N/H\rightarrow N/K$ a left inverse 
of the canonical projection $N/K\rightarrow N/H$. Then $\Q Ge(G,H,K)$ is isomorphic to $M_n(\Q(\zeta_k)*_{\tau}^{\alpha} N/H)$ and the action and
twisting are given by
\begin{eqnarray*}
\alpha_{nH}(\zeta_k) &=& \zeta_k^i, \mbox{ if } yK^{\phi(nH)}=y^iK ,\\
\tau(nH,n'H) &=& \zeta_k^j, \mbox{ if }  \phi(nn'H)\inv\phi(nH)\phi(n'H)=y^jK,
\end{eqnarray*}
for $nH,n'H\in N/H$ and integers $i$ and $j$.
\end{theorem}

Note that the action $\alpha$ of the crossed product $\Q(\zeta_k)*_{\tau}^{\alpha} N/H$ in Theorem \ref{SSP} is faithful. Therefore the crossed product $\Q(\zeta_k)*_{\tau}^{\alpha} N/H$ can be described as a classical crossed product $(\Q(\zeta_k)/F,\tau)$, where $F$ is the center of the algebra, which is determined by the Galois action $\alpha$.

A subring $\O$ of a finite dimensional $\Q$-algebra $A$ is called an order if it is a finitely generated $\Z$-module such
that $\Q\O=A$. For example, $\Z G$ is an order in $\Q G$ when $G$ is finite. It is known that the intersection of two orders
in $A$ is again an order in $A$ and that, if $\O_1\subseteq \O_2$ are orders in $A$, then the index of their unit groups
$[\U(\O_2):\U(\O_1)]$ is finite \cite{Sehgal1993}. Therefore, for two arbitrary orders $\O_1,\O_2$ in $A$, we have that
$[\U(\O_2):\U(\O_1\cap \O_2)]$ is finite, in other words $\U(\O_1)$ and $\U(\O_2)$ are commensurable. Moreover, the group of
units of an order in a finite dimensional semisimple algebra is finitely generated \cite{1943Siegel,1962BorelHarishChandra}.
Hence the unit group of $\Z G$ is finitely generated if $G$ is a finite group. Recall that in a finitely generated abelian
group replacing generators by powers of themselves yields generators for a subgroup of finite index. We will use these
properties several times in the proofs without explicitly mentioning.

\section{A generalization of the Jespers-Parmenter-Sehgal Theorem}

In this section we prove a generalization of a theorem of Jespers, Parmenter and Sehgal while also avoiding the use of
$K$-theory. For a finite abelian-by-supersolvable group $G$ such that every cyclic subgroup of order not a divisor of 4 or 6,
is subnormal in $G$, the detailed description of the primitive central idempotents of $\Q G$ and the Bass-Milnor Theorem
allow us to show that the group generated by the set of Bass units of $\Z G$ contains a subgroup of finite index in $\mathcal{Z}(\U(\Z G))$.
Furthermore, one obtains a description for the generators of this subgroup.

In order to do this, we first need a new construction for central units based on Bass units in the integral group ring $\Z
G$. 

The idea originates from \cite{JesParSeh1996}, in which the authors constructed central units in $\Z G$ based on Bass units $b\in \Z G$ for finite nilpotent groups $G$. We denote by $Z_i$ the $i$-th center, i.e. $Z_0=1$ and $Z_i\unlhd G$ is defined such that $Z_i/Z_{i-1}=\mathcal{Z}(G/Z_{i-1})$. Since $G$ is nilpotent, $Z_n=G$ for some $n$. For any $g\in G$ and a Bass unit $b$ based on $g$, put $b_{(1)}=b$, and, for $2\leq i \leq n$, put $$b_{(i)}=\prod_{h\in Z_i}b_{(i-1)}^h.$$ By induction, $b_{(i)}$ is independent of the order of the conjugates in the product expression and $b_{(i)}$ is central in $\Z\GEN{Z_i,g}$, since for every $h\in Z_i$ and for every $i$ there exists $x\in Z_{i-1}$ such that $hg=xgh$ and $\GEN{Z_{i-1},g}\unlhd\GEN{Z_i,g}$. In particular, $b_{(n)} \in \mathcal{Z}(\U(\Z G))$.

Note that the previous construction can be modified and improved by considering the subnormal series $\GEN{g}\unlhd\GEN{Z_1,g}\unlhd\cdots\unlhd 
\GEN{Z_n,g}=G$ and taking in each step conjugates in a transversal for $Z_i$ in $Z_{i-1}$. Then, the two constructions differ by a power. The
constructions remain valid when starting with an arbitrary unit $u$ in $\Z G$ with support in an abelian subgroup. 

We now generalize this construction to a bigger class of groups $G$. Throughout $G$ will be a finite abelian-by-supersolvable group such that every cyclic 
subgroup of order not a divisor of 4 or 6, is subnormal in $G$. It is clear that this class of groups contains the finite nilpotent groups, the
dihedral groups $D_{2n}=\GEN{x,y\mid x^n=1=y^2,yxy=x\inv}$ and the generalized quaternion groups $Q_{2n}=\GEN{x,y\mid x^{2n}=1=y^4,x^n=y^2,y\inv
xy=x\inv}$. 

Let $u\in \U(\Z\GEN{g})$, for $g\in G$ of order not a divisor of 4 or 6. We consider a subnormal series $\NN:N_0=\GEN{g}\lhd N_1 \lhd N_2 \lhd \cdots 
\lhd N_m=G$. Now define $\cc{\NN}_0(u)=u$ and $$\cc{\NN}_i(u)=\prod_{h\in T_i}\cc{\NN}_{i-1}(u)^h,$$ where $T_i$ is a transversal for $N_i$ in
$N_{i-1}$. We will prove that this construction is well defined by proving the following three properties.
\begin{lemma}\label{lemmac}
Let $g\in G, u\in\U(\Z\GEN{g})$, $\NN$, $N_i$ and $T_i$ be as above. We have
\begin{enumerate}[label=\rm(\Roman{*}), ref=\Roman{*}]
 \item \label{eq1} $\forall x\in N_i: \cc{\NN}_{i-1}(u)^x\in \Z N_{i-1}$, \\
 \item \label{eq2} $\forall x\in N_{i-1}: \cc{\NN}_{i-1}(u)^x=\cc{\NN}_{i-1}(u)$, \\
 \item \label{eq3} $\cc{\NN}_i(u)$ is independent of the transversal $T_i$. 
\end{enumerate}
\end{lemma}
\begin{proof}
It is easy to see that equation (\ref{eq2}) implies (\ref{eq3}). Hence it is sufficient to prove (\ref{eq1}) and (\ref{eq2}).

We prove these by induction on $i$. First assume $i=1$. Then (\ref{eq1}) and (\ref{eq2}) are trivial since the support of $u$ is contained in 
$\GEN{g}=N_0\lhd N_1$.

Now assume the formulas hold for $i-1$. Let $x\in N_i$. Then $\cc{\NN}_{i-1}(u)^x=\prod_{h\in T_{i-1}}\cc{\NN}_{i-2}(u)^{hx}$. By the induction
hypothesis we have that $\cc{\NN}_{i-2}(u)^h\in \Z N_{i-2}$ and since $N_{i-2}\lhd N_{i-1}\lhd N_i$, also $\cc{\NN}_{i-2}(u)^{hx}\in \Z N_{i-1}$,
which proves (\ref{eq1}).

Now let $x\in N_{i-1}$. Then $\cc{\NN}_{i-1}(u)^x=\prod_{h\in T_{i-1}}\cc{\NN}_{i-2}(u)^{hx}=\prod_{h'\in T_{i-1}x}\cc{\NN}_{i-2}(u)^{h'}$. 
$T_{i-1}x$ remains a transversal for $N_{i-1}$ in $N_{i-2}$. Hence, by the induction hypothesis on (\ref{eq3}), the latter equals $\cc{\NN}_{i-1}(u)$
and we have proved (\ref{eq2}).
\end{proof}

By equations (\ref{eq1}) and (\ref{eq2}) we have that the construction is independent of the order of the conjugates in the product expression. 
Furthermore, $\cc{\NN}_m(u)$, the final step in our construction, is a central unit in $\Z G$, which we will simply denote by $\cc{\NN}(u)$.

\begin{theorem}\label{nilpotent+}
Let $G$ be a finite abelian-by-supersolvable group such that every cyclic subgroup of order not a divisor of 4 or 6, is subnormal in $G$. Then 
the group generated by the Bass units of $\Z G$ contains a subgroup of finite index in $\mathcal{Z}(\U(\Z G))$.
\end{theorem}
\begin{proof}
We argue by induction on the order of the group $G$. For $G=1$ the result is clear. So assume now that the result holds for groups of order strictly 
less than the order of $G$.

Because of the Bass-Milnor Theorem \cite{bass1966}, we can assume that $G$ is non-abelian. Write $\Q G=\Q
G(1-\suma{G'})\bigoplus \Q G\suma{G'} $, where $G'$ is the commutator subgroup of $G$. It is well known that $\Q
G(1-\suma{G'})$ is a direct sum of non-commutative simple rings and $\Q G\suma{G'}\simeq \Q(G/G')$ is a commutative group
ring. Hence, each $z\in \mathcal{Z}(\U(\Z G))$ can be written as $z=z'+z''$, with $z'\in \mathcal{Z}(\U(\Z G (1-\suma{G'})))$
and $z''\in \U(\Z G\suma{G'})$. Note that $z'z''=0=z''z'$. We will prove that some positive power of $z$ is a product of Bass
units. Since $z$ is an arbitrary element of the finitely generated abelian group $\mathcal{Z}(\U(\Z G))$, the result follows.

First we focus on the commutative component. Since $G/G'$ is abelian, it follows from the Bass-Milnor Theorem that the
Bass units of $\Z (G / G')$ generate a subgroup of finite index in $\U(\Z(G/G'))$. A power of each Bass unit of $\Z (G/G')$ is
the natural image of a Bass unit of $\Z G$. Hence, we get that $\overline{z''}^m=\prod_{i=1}^r\overline{b_i}$ for some
positive integer $m$ and some Bass units $b_i$ in $\Z G$, where we denote the natural image of $x\in \Z G$ in $\Z (G/G')$ by 
$\overline{x}$. It is well-known and easy to verify that $\Bass{k}{m}{g}$ has finite order if and only if $k\equiv \pm 1 \mod
|g|$. In particular, there is a Bass unit based on $g\in G$ of infinite order if and only if the order of $g$ is not a
divisor of 4 or 6. Hence we can assume that each $b_i$ is based on an element of order not a divisor of 4 or 6.

By the assumptions on $G$, we can construct central units in $\Z G$ which project to some power of a $\overline{b_i}$ in $\Z
(G/G')$. Indeed, each $\cc{\NN_i}(b_i)$ is central in $\Z G$, where $\NN_i$ is a subnormal series from $\GEN{g_i}$ to $G$
when $b_i$ is based on $g_i$. Since $\Z (G /G')$ is commutative, the natural image of $\cc{\NN_i}(b_i)$ is a power of
$\overline{b_i}$, say $\overline{b_i}^{m_i}$. Hence $\overline{z''}^{m\cdot\lcm(m_i:1\leq i\leq r)}=\prod_{i=1}^r
\overline{\cc{\NN_i}(b_i)}^{\lcm(m_i:1\leq i\leq r)/m_i}$. Hence one may assume there exists some positive integer $m'$ such
that $\overline{z''}^{m'}=\prod_{j=1}^s \overline{\cc{\NN_j}(b_j)}$, where $b_j$ runs through a set of Bass units of $\Z G$
with possible repetition. Therefore, $z^{m'} (\prod_{j=1}^s \cc{\NN_j}(b_j))\inv=z''' + \suma{G'}$, with $z'''\in
\mathcal{Z}(\U(\Z G(1-\suma{G'})))$.

Since $G$ is abelian-by-supersolvable and hence also strongly monomial, we know (see the Preliminaries) that $\Q
G(1-\suma{G'})=\bigoplus_{(H,K)} \Q Ge(G,H,K),$ where $(H,K)$ runs through a complete and non-redundant set of strong Shoda
pairs of $G$ with $\Q Ge(G,H,K)$ not commutative. Note that in particular $H\neq G$ for each such strong Shoda pair.

Let $(H,K)$ be such a strong Shoda pair of $G$. Then it is also a strong Shoda pair of $H$ and $e(H,H,K)=\varepsilon(H,K)$
is a primitive central idempotent of $\Q H$. Since $|H|<|G|$, the induction hypothesis yields that there exists a subgroup
$A_1$ in the group generated by the Bass units in $\Z H$ such that $A_1$ is of finite index in $\mathcal{Z}(\U(\Z H))$.
Clearly, $\Z H\subseteq \bigoplus_e \Z He$, where $e$ runs through all primitive central idempotents of $\Q H$. As both $\Z
H$ and $\bigoplus_e \Z He$ are $\Z$-orders in $\Q H$, we have that $\mathcal{Z}(\U(\Z H))$ is of finite index in
$\mathcal{Z}(\U(\bigoplus_e \Z He))$. Hence, $A_1$ is of finite index in $\mathcal{Z}(\U(\bigoplus_e \Z He))$. Since
$\Z(1-\varepsilon(H,K))+\Z H\varepsilon(H,K)\subseteq \mathcal{Z}(\bigoplus_e \Z H e)$, we thus get that $A=A(H,K)=A_1\cap
\left(\Z(1-\varepsilon(H,K))+\Z H\varepsilon(H,K)\right)$ is of finite index in $\U(\Z(1-\varepsilon(H,K))\bigoplus \Z
H\varepsilon(H,K))$, and each element of $A$ is a product of Bass units of $\Z H$.

From Theorem \ref{SSP} we know that $\Q Ge(G,H,K)\simeq M_{[G:N_G(K)]}(\Q H\varepsilon(H,K)*(N_G(K)/H))$ and its center
consists of the scalar matrices with diagonal entry in $(\Q H\varepsilon(H,K))^{N_G(K)/H}$, the fixed subfield of $\Q
H\varepsilon(H,K)$ under the action of $N_G(K)/H$. Since $\U(\Z H\varepsilon(H,K))$ is a finitely generated abelian group, it
is easy to verify that $\{\prod_{n\in N_G(K)} u^n \mid u\in \U(\Z H\varepsilon(H,K)) \}$ is of finite index in $\U((\Z
H\varepsilon(H,K))^{N_G(K)/H})$.  

Next note that if $\alpha=1-\varepsilon(H,K)+\beta\varepsilon(H,K)\in A$, with $\beta\in \Z H$, then 
$\alpha^n=1-\varepsilon(H,K)+\beta^n\varepsilon(H,K)$, for $n\in N_G(K)$. Hence, $\alpha$ and $\alpha^n$ commute and thus the product $\prod_{n\in
N_G(K)} \alpha^n$ is independent of the order of its factors. It follows from the previous that $B=B(H,K)=\{\prod_{n\in N_G(K)}\alpha^n \mid \alpha
\in A\cap (1-\varepsilon(H,K)+\Z H\varepsilon(H,K))\}$ is a subgroup of finite index of $\U(\Z(1-\varepsilon(H,K))+(\Z H\varepsilon(H,K))^{N_G(K)/H})$
and the elements of $B$ are products of Bass units in $\Z H$.

Let $\gamma=1-\varepsilon(H,K)+\delta\in B$, with $\delta\in(\Z H\varepsilon(H,K))^{N_G(K)/H}$. Let $T$ be a right transversal of $N_G(K)$ in $G$. 
Since $\varepsilon(H,K)^t\varepsilon(H,K)^{t'}=0$ for different $t,t'\in T$, we get that $\gamma^t$ and $\gamma^{t'}$ commute and $\prod_{t\in
T}\gamma^t=1-e(G,H,K)+\sum_{t\in T}\delta^t\in 1-e(G,H,K)+\Z G e(G,H,K)$. Clearly, $\prod_{t\in T}\gamma^t$ corresponds to a central matrix in $\Q
Ge(G,H,K)$ with diagonal entry in $(\Z H\varepsilon(H,K))^{N_G(K)/H}$. From the previous it follows that $C=C(H,K)=\{\prod_{t\in T}\gamma^t \mid
\gamma\in B\}$ is a subgroup of finite index in $\mathcal{Z}\left(\U(\Z(1-e(G,H,K))+ \Z Ge(G,H,K))\right)$. As each $\gamma\in B$ is a product of Bass
units in $\Z H$, so is $\prod_{t\in T}\gamma^t$ a product of Bass units in $\Z G$.

We can now finish the proof as follows. Write the central unit
$$z'''+ \suma{G'} = \sum_{(H,K)} z'''e(G,H,K) +\suma{G'}
= \prod_{(H,K)}(1-e(G,H,K) + z'''e(G,H,K)),$$
where $(H,K)$ runs through a complete and non-redundant set of strong Shoda pairs of $G$ so that $\Q Ge(G,H,K)$ is not commutative and 
$\sum_{(H,K)}e(G,H,K)=1-\suma{G'}$. Because of the construction of $C(H,K)$, there exists a positive integer $m''$ so that
$$(1-e(G,H,K)+z'''e(G,H,K))^{m''}\in C(H,K)$$ for each $(H,K)$. Hence $(z'''+\suma{G'})^{m''}$ and thus also $z^{m'm''}=\left(\prod_{j=1}^s
\cc{\NN_j}(b_j)\right)^{m''}(z'''+\suma{G'})^{m''}$ is a product of Bass units in $\Z G$.
\end{proof}

\begin{remark}\label{remark_construction}
Only one argument in the proof of Theorem \ref{nilpotent+} makes use of the assumption that every cyclic subgroup of order not a divisor of 4 or 6, 
is subnormal in $G$. It is needed to produce a central unit as a product of conjugates of a Bass unit $b$. For that we use the construction
$\cc{\NN}(b)$. It is not clear to us whether an alternative construction exists for other classes of groups, even for metacyclic groups this is
unknown.

At first sight, it looks like we do not use the properties of abelian-by-supersolvable groups, except for the fact that these groups are strongly 
monomial and hence we know an explicit description of the Wedderburn components. However, we can not generalize the proof to strongly monomial groups
since we use an induction hypothesis on subgroups and, unlike the class of abelian-by-supersolvable groups, the class of strongly monomial groups is
not closed under subgroups. 
\end{remark}

\begin{corollary}\label{cor_JPS+}
 Let $G$ be a finite abelian-by-supersolvable group such that every cyclic subgroup of order not a divisor of 4 or 6, is subnormal in $G$. For each 
such cyclic subgroup $\GEN{g}$, fix a subnormal series $\NN_g$ from $\GEN{g}$ to $G$. Then $$\GEN{\cc{\NN_g}(b_g) \mid b_g \mbox{ a Bass unit based
on } g, g\in G}$$ is of finite index in $\mathcal{Z}(\U(\Z G))$.
\end{corollary}
\begin{proof}
Because $\mathcal{Z}(\U(\Z G))$ is finitely generated and using Theorem \ref{nilpotent+}, it is sufficient to show that if
$u=b_1b_2\cdots b_m\in \mathcal{Z}(\U(\Z G))$, with each $b_i$ a Bass unit based on $g_i\in G$, then there exists a positive
integer $l$ so that $u^l$ is a product of $\cc{\NN_g}(b_g)$'s, with $b_g$ a Bass unit based on $g\in G$. In order to prove
this, for each primitive central idempotent $e$ of $\Q G$, write $\Q Ge=M_{n_e}(D_e)$, with $n_e$ a positive integer and
$D_e$ a division algebra. If $\O_e$ is an order in $D_e$, then we have that $\U(\Z G)\cap \prod_e \GL_{n_e}(\O_e)$ is of
finite index in $\U(\Z G)$, where $\GL_{n_e}(\O_e)$ denotes the group of invertible elements in $M_{n_e}(\O_e)$. It is well
known and easy to verify that the central matrices in $\SL_{n_e}(\O_e)$, consisting of the reduced norm one matrices in
$M_{n_e}(\O_e)$, are torsion.

Now, let $u=b_1b_2\cdots b_m\in \mathcal{Z}(\U(\Z G))$, with each $b_i$ a Bass unit based on $g_i\in G$. Then there exists a positive integer $m'$ 
such that $u^{m'}, \left(\prod_{i=1}^m \cc{\NN_{g_i}}(b_i)\right)^{m'}\in \prod_e \GL_{n_e}(\O_e)$. Let $k_i$ be a positive integer so that each
$\cc{\NN_{g_i}}(b_i)$ is a product of $k_i$ conjugates of $b_i$. Then $\cc{\NN_{g_i}}(b_i)e$ and ${b_i}^{k_i}e$ have the same reduced norm. Hence,
$$u^{km'}\prod_{i=1}^m \cc{\NN_{g_i}}(b_i)^{-m'k/k_i}e\in \SL_{n_e}(\O_e)\cap \mathcal{Z}(\GL_{n_e}(\O_e)),$$ for $k=\lcm(k_i:1\le i\le m)$ and thus
$u^{km'}\prod_{i=1}^m \cc{\NN_{g_i}}(b_i)^{-m'k/k_i}e$ is a torsion element in $\mathcal{Z}(\GL_{n_e}(\O_e))$. Consequently,
$$\left(u^{km'}\prod_{i=1}^m \cc{\NN_{g_i}}(b_i)^{-m'k/k_i}\right)^{m''}=1$$ for some positive integer $m''$, i.e. $u^{km'm''}\in
\GEN{\cc{\NN_{g}}(b_g)\mid b_g \mbox{ a Bass unit based on } g, g\in G}$. 
\end{proof}

For finite nilpotent groups of class $n$, we can always take the subnormal series $\NN_g:\GEN{g}\unlhd\GEN{Z_1,g}\unlhd\cdots\unlhd \GEN{Z_n,g}=G$. 
Since both constructions $\cc{\NN_g}(b)$ and $b_{(n)}$ only differ on a power, we can deduce the Jespers-Parmenter-Sehgal result.

\begin{corollary}(Jespers-Parmenter-Sehgal)\label{cor_JPS}
 Let $G$ be a finite nilpotent group of class $n$. Then $$\GEN{b_{(n)} \mid b \mbox{ a Bass unit in } \Z G}$$ is of finite index in 
$\mathcal{Z}(\U(\Z G))$.
\end{corollary}

\section{Reducing to a basis of products of Bass units for $\mathcal{Z}(\U(\Z G))$}

In this section, we obtain a basis formed by products of Bass units of a free abelian subgroup of finite index in $\mathcal{Z}(\U(\Z G))$, 
for $G$ a finite abelian-by-supersolvable group $G$ such that every cyclic subgroup of order not a divisor of 4 or 6, is subnormal in $G$.

First, we need some properties of our construction of central units.
\begin{lemma}\label{prod_conj}
 Let $G$ be a finite abelian-by-supersolvable group such that every cyclic subgroup of order not a divisor of 4 or 6, is subnormal in $G$. 
Let $u,v$ be units in $\Z\GEN{g}$ for $g\in G$ and let $\NN$ be a subnormal series $N_0=\GEN{g}\lhd N_1\lhd \cdots \lhd N_m=G$. Assume $h\in G$ and
denote by $\NN^h$ the $h$-conjugate of the series $\NN$, i.e. $\NN^h:N_0^h=\GEN{g^h}\lhd N_1^h\lhd\cdots\lhd N_m^h=G$. Then
\begin{enumerate}[label=\rm(\Alph{*}), ref=\Alph{*}]
\item \label{cprod}$\cc{\NN}(uv)=\cc{\NN}(u)\cc{\NN}(v)$, and
\item \label{cconj}$\cc{\NN^h}(u^h)=\cc{\NN}(u)$. 
\end{enumerate}
\end{lemma}
\begin{proof}
Let $u,v\in \Z\GEN{g}$. Then clearly $\cc{\NN}_0(uv)=uv=\cc{\NN}_0(u)\cc{\NN}_0(v)$. By an induction argument on $i$, we now get that 
$$\cc{\NN}_i(uv)=\prod_{x\in T_i}\cc{\NN}_{i-1}(uv)^x=\prod_{x\in T_i}\cc{\NN}_{i-1}(u)^x\cc{\NN}_{i-1}(v)^x=\cc{\NN}_i(u)\cc{\NN}_i(v),$$ for $i\geq
1$, since $\cc{\NN}_{i-1}(u)^x$ and $\cc{\NN}_{i-1}(v)^x$ commute by properties (\ref{eq1}) and (\ref{eq2}). This proves (\ref{cprod}).

Let $u\in \Z\GEN{g}$ and $h\in G$. We prove that $\cc{\NN^h}_i(u^h)=\cc{\NN}_i(u)^h$ by induction on $i$. For $i=0$ we have
$\cc{\NN^h}_0(u^h)=u^h=\cc{\NN}_0(u)^h$. Let $i\geq 1$, then by the induction hypothesis $$\cc{\NN^h}_i(u^h)=\prod_{x\in
T_i^h}\cc{\NN^h}_{i-1}(u^h)^x=\prod_{x\in T_i^h}\cc{\NN}_{i-1}(u)^{hx}=\prod_{y\in
T_i}\cc{\NN}_{i-1}(u)^{yh}=\cc{\NN}_i(u)^h.$$ 
\end{proof}

Let $G$ be a group. If $g\in G$, we denote by $C_g$ the conjugacy class of $g$ in $G$. $\R$-classes and $\Q$-classes are
a generalization of this. For a given element $g$ in a group $G$ of exponent $e$, the $\R$-class of $g$ is defined as the
union $C_g\cup C_{g\inv}$, and the $\Q$-class of $g$ is defined as the union $\bigcup_{\gcd(r,e)=1} C_{g^r}$. The number of 
$\Q$-classes of a group coincides with the number of conjugacy classes of cyclic subgroups of $G$ and, by a result of Artin, this
number coincides with the number of irreducible rational characters of $G$, i.e. the number of simple components of $\Q G$
\cite[Cor. 39.5, Th. 42.8]{1962CurtisReiner}.

Let $g\in G$ and define $$S_g=\{l\in \U(\Z_{|g|}) : g \mbox{ is conjugate with } g^l \mbox{ in } G\}.$$ In other words, $S_g$ is the image of the 
homomorphism
 $$N_G(\GEN{g})\rightarrow \U(\Z_{|g|}):h\mapsto l_h $$
where $l_h$ is the unique element of $\U(\Z_{|g|})$ such that $g^h=g^{l_h}$. The kernel of this homomorphism is $\Cen_G(g)$. We denote 
$\overline{S_g}=\GEN{S_g,-1}$ and we always assume that transversals of $\overline{S_g}$ in $\U(\Z_{|g|})$ contain the identity $1$.

\begin{theorem}\label{Basis}
Let $G$ be a finite abelian-by-supersolvable group such that every cyclic subgroup of order not a divisor of 4 or 6, is subnormal in $G$. Let $R$ 
denote a set of representatives of $\Q$-classes of $G$. For $g\in R$ choose a transversal $T_g$ of $\overline{S_g}$ in $\U(\Z_{|g|})$ containing 1 and
for every $k\in T_g\setminus\{1\}$ choose an integer $m_{k,g}$ with $k^{m_{k,g}}\equiv 1 \mod |g|$. For every $g\in R$ of order not a divisor of 4 or
6, choose a subnormal series $\NN_g$ from $\GEN{g}$ to $G$. Then
    $$\left\{\cc{\NN_g}(\Bass{k}{m_{k,g}}{g}): g\in R, k\in T_g\setminus\{1\}\right\}$$
is a basis for a free abelian subgroup of finite index in $\mathcal{Z}(\U(\Z G))$.
\end{theorem}
\begin{proof}
For every $g\in R$ of order not a divisor of 4 or 6, we choose a subnormal series $\NN_g$ from $\GEN{g}$ to $G$. Then for each $h\in G$ of order 
not a divisor of 4 or 6, we agree to choose the subnormal series $\NN_h$ to be the $x$-conjugate of $\NN_g$ when $h=x\inv g^ix$ with $g\in R$ and $i$
coprime to the order of $g$.

By Corollary \ref{cor_JPS+}, the set $$B_1=\{\cc{\NN_h}(\Bass{k}{m}{h}) \mid h\in G, k,m\in \N, k^m\equiv 1 \mod |h| \}$$ generates a subgroup of 
finite index in $\mathcal{Z}(\U(\Z G))$. Let $t=\varphi(|G|)$. We first prove that $$B_2=\left\{\cc{\NN_g}(\Bass{k}{t}{g})\mid g\in R, k\in
T_g\setminus\{1\}\right\}$$ generates a subgroup of finite index in $\mathcal{Z}(\U(\Z G))$. To do so we sieve gradually the list of units in $B_1$,
keeping the property that the remaining units still generate a subgroup of finite index in $\mathcal{Z}(\U(\Z G))$, until the remaining units are the
elements of $B_2$. 

By equation (\ref{Basseq1}), to generate $B_1$ it is enough to use the Bass units of the form $\Bass{k}{m}{h}$ with $h\in G$, $1\le k < |h|$ and 
$k^m \equiv 1 \mod |h|$. Hence one can assume that $k\in \U(\Z_{|h|})$.

By (\ref{Basseq6}), for every Bass unit $\Bass{k}{m}{h}$ we have $\Bass{k}{m}{h}^i = \Bass{k}{t}{h}^j$ for some positive integers $i$ and $j$. 
Thus, by (\ref{cprod}), units of the form $\cc{\NN_h}(\Bass{k}{t}{h})$ with $k\in \U(\Z_{|h|})$ generate a subgroup of finite index in
$\mathcal{Z}(\U(\Z G))$.

By the definition of a $\Q$-class, we know that each $h\in G$ is conjugate to some $g^i$, for $g\in R$ and $(i,|g|)=1$. Hence, by (\ref{Basseq3}), 
(\ref{cprod}) and (\ref{cconj}), we can reduce further the list of generators by taking only Bass units based on elements of $R$.

By (\ref{Basseq4}), we can exclude $k=1$ and still generate a subgroup of finite index in $\mathcal{Z}(\U(\Z G))$.

Let $g\in G$ be of order $n$. We claim that if $l\in \overline{S_g}$ and $k\in \U(\Z_n)$ then $\cc{\NN_g}(\Bass{l}{t}{g^k})$ has finite order. 
As $\Bass{n-l}{t}{g^k}=\Bass{l}{t}{g^k}g^{-lkt}$, by (\ref{Basseq8}), we may assume without loss of generality that $l\in S_g$. By (\ref{Basseq3}),
(\ref{cprod}) and (\ref{cconj}) we have
    $$\cc{\NN_g}(\Bass{l^{i+1}}{t}{g^k}) = \cc{\NN_g}(\Bass{l}{t}{g^k}) \cc{\NN_g}(\Bass{l^i}{t}{g^{kl}}) =  \cc{\NN_g}(\Bass{l}{t}{g^k}) 
\cc{\NN_g}(\Bass{l^i}{t}{g^{k}}).$$
Then, arguing inductively we deduce that
    $$\cc{\NN_g}(\Bass{l}{t}{g^k})^i = \cc{\NN_g}(\Bass{l^i}{t}{g^k}),$$
and in particular $\cc{\NN_g}(\Bass{l}{t}{g^k})^t=\cc{\NN_g}(\Bass{l^t}{t}{g^k})=\cc{\NN_g}(\Bass{1}{t}{g^k})=1$, by (\ref{Basseq1}) and 
(\ref{Basseq4}). This proves the claim.

With $g$ and $n$ as above, every element of $\U(\Z_n)$ is of the form $kl$ with $k\in T_g$ and $l\in \overline{S_g}$. Using (3) again we have 
$\Bass{kl}{t}{g}=\Bass{k}{t}{g}\Bass{l}{t}{g^k}$. By the previous paragraph, $\cc{\NN_g}(\Bass{l}{t}{g^k})$ has finite order. Hence we can reduce the
generating system and take only $k\in T_g\setminus\{1\}$.

The remaining units are exactly the elements of $B_2$. Thus $\GEN{B_2}$ has finite index in $\mathcal{Z}(\U(\Z G))$, as desired.

Let $B=\{\cc{\NN_g}(\Bass{k}{m_{k,C}}{g}) : g\in R, k\in T_G\setminus\{1\}\}$. Using (\ref{Basseq6}) once more, we deduce that $\GEN{B}$ has 
finite index in $\mathcal{Z}(\U(\Z G))$, since $\GEN{B_2}$ does. 

To finish the proof we need to prove that the elements of $B$ are multiplicatively independent. To do so, it is enough to show that the rank 
of $\mathcal{Z}(\U(\Z G))$ coincides with the cardinality of $B$. It is easy to see that $|B|=\left(\sum_{g\in R}|T_g|\right)-|R|$ and $|R|$ equals
the number of $\Q$-classes. By construction, $[\U(\Z_{|g|}):S_g]$ equals the number of conjugacy classes contained in the $\Q$-class of $g$.
Furthermore,  $[\overline{S_g}:S_g]=1$ when $g$ is conjugated to $g\inv$ and $[\overline{S_g}:S_g]=2$ when $g$ is not conjugated to $g\inv$. Therefore
$|T_g|=[\U(\Z_{|g|}):\overline{S_g}]$ is exactly the number of $\R$-classes contained in the $\Q$-class of $g$. Hence $|B|$ equals the number of
$\R$-classes minus the number of $\Q$-classes in $G$. By a result in \cite{2005RitterSehgal,Ferraz2004}, this number coincides with the rank of
$\mathcal{Z}(\U(\Z G))$ and the proof is finished.
\end{proof}

\section{Generalization to strongly monomial groups}

As mentioned in Remark \ref{remark_construction}, it is not known whether Theorem \ref{nilpotent+} remains valid for other classes of groups, 
including metacyclic groups. In this section we construct generalized Bass units and show that the group they generate contains a subgroup of finite
index in the central units of the integral group ring $\Z G$ for finite strongly monomial groups $G$. This generalizes Corollary 2.3 in
\cite{JesPar2011} on generators for central units of the integral group ring of a finite metabelian group.

Let $R$ be an associative ring with identity. Let $x$ be a torsion unit of order $n$. Let $C_n=\GEN{g}$, a cyclic group of order $n$. Then 
the map $g\mapsto x$ induces a ring homomorphism $\Z\GEN{g}\rightarrow R$. If $k$ and $m$ are positive integers with $k^m\equiv 1\mod n$, then the
element $$\Bass{k}{m}{x}=(1+x+\dots + x^{k-1})^{m}+\frac{1-k^m}{n}(1+x+\dots+x^{n-1})$$ is a unit in $R$ since it is the image of a Bass unit in
$\Z\GEN{g}$.

In particular, if $G$ is a finite group, $M$ a normal subgroup of $G$, $g\in G$ and $k$ and $m$ positive integers such that $\gcd(k,|g|)=1$ and 
$k^m \equiv 1 \mod |g|$. Then we have
$$\Bass{k}{m}{1-\suma{M}+g\suma{M}} = 1-\suma{M}+\Bass{k}{m}{g}\suma{M}.$$
Observe that any element $b=\Bass{k}{m}{1-\suma{M}+g\suma{M}}$ is an invertible element of $\Z G(1-\suma{M})+\Z G \suma{M}$. As this is an order 
in $\Q G$, there is a positive integer $n$ such that $b^n\in \U(\Z G)$. Let $n_b$ denote the minimal positive integer satisfying this condition. Then
we call the element $$\Bass{k}{m}{1-\suma{M}+g\suma{M}}^{n_b}=\Bass{k}{mn_b}{1-\suma{M}+g\suma{M}}$$ a generalized Bass unit based on $g$ and $M$ with
parameters $k$ and $m$. Note that we obtain the classical Bass units when $M=1$.

\begin{theorem}\label{strongmonomial}
Let $G$ be a finite strongly monomial group. Then the group generated by the generalized Bass units $b^{n_b}$ with 
$b=\Bass{k}{m}{1-\suma{H'}+h\suma{H'}}$ for a strong Shoda pair $(H,K)$ of $G$ and $h\in H$, contains a subgroup of finite index in 
$\mathcal{Z}(\U(\Z G))$.
\end{theorem}
\begin{proof}
 Let $(H,K)$ be a strong Shoda pair of $G$. Since $H/H'$ is abelian, it follows from the Bass-Milnor Theorem that the Bass units of $\Z (H/H')$ 
generate a subgroup of finite index in the group of (central) units of $\Z (H/H')\simeq \Z H\suma{H'}$. A power of each Bass unit of $\Z(H/H')$ is the
natural image of a Bass unit in $\Z H$. Hence the group generated by units of the form $b=\Bass{k}{m}{1-\suma{H'}+h\suma{H'}}$, with $h\in H$, is of
finite index in $\U(\Z(1-\suma{H'})+\Z H\suma{H'})$. Then the group generated by the generalized Bass units $b^{n_b}$ is still of finite index in
$\U(\Z(1-\suma{H'})+\Z H\suma{H'})$. Let $A_1$ denote this subgroup. Note that $A_1$ is central in $\Z H$.

 Since $H'\subseteq K\subseteq H$, we know that $\varepsilon(H,K)\in \Z H\suma{H'}$ and hence $\Z(1-\varepsilon(H,K))+\Z H\varepsilon(H,K)\subseteq 
\Z(1-\suma{H'})+\Z H\suma{H'}$. Therefore $A=A_1\cap (\Z(1-\varepsilon(H,K))+\Z H\varepsilon(H,K))$ is of finite index in
$\mathcal{Z}(\U(\Z(1-\varepsilon(H,K))+\Z H\varepsilon(H,K)))$.

 As in the proof of Theorem \ref{nilpotent+}, one obtains that $B=\{\prod_{g\in N_G(K)}\alpha^g \mid \alpha \in A\cap 
(1-\varepsilon(H,K)+\Z H\varepsilon(H,K))\}$ generates a subgroup of finite index in $\U(\Z(1-\varepsilon(H,K))+(\Z H\varepsilon(H,K))^{N_G(K)/H})$.
Furthermore, if $T$ is a transversal of $N_G(H)$ in $G$ then $C=\{\prod_{t\in T}\gamma^t \mid \gamma\in B\}$ generates a subgroup of finite index in
$$\mathcal{Z}\left(\U(\Z(1-e(G,H,K))+ \Z Ge(G,H,K))\right).$$ The products involved in the definition of $B$ do not depend on the order of
multiplication because the factors $\alpha^g$ belong to $\Z(1-\varepsilon(H,K))+\Z H \varepsilon(H,K)$ and this is a commutative ring. To prove that
the product in the definition of $C$ is independent on the order of the product observe that if $\gamma\in B$ and $t_1,t_2\in G$ then
$\gamma=1-\varepsilon(H,K)+\gamma_1 \varepsilon(H,K)$ for some $\gamma_1\in \Z H$. If $t_1\ne t_2$ then $\varepsilon(H,K)^{t_1}
\varepsilon(H,K)^{t_2}=0$, because $(H,K)$ is a strong Shoda pair. Using this it is easy to see that $\gamma^{t_1}$ and $\gamma^{t_2}$ commute.

Take now an arbitrary central unit $u$ in $\mathcal{Z}(\U(\Z G))$. Then we can write this element as follows
$$u=\sum_{(H,K)}ue(G,H,K)=\prod_{(H,K)} (1-e(G,H,K)+ue(G,H,K)),$$ where $(H,K)$ runs through a complete and non-redundant set
of strong Shoda pairs of $G$. Note that conjugates of $b^{n_b}=\Bass{k}{m}{1-\suma{H'}+h\suma{H'}}$ are again of this form
since conjugates of Bass units are again Bass units and because $(H^g,K^g)$ is a strong Shoda pair of $G$ if $(H,K)$ is a
strong Shoda pair of $G$ for $g\in G$. Hence the result follows from the previous paragraph.
\end{proof}

Theorem \ref{strongmonomial} extends the Bass-Milnor Theorem because for abelian groups the generalized Bass units are precisely the Bass units.

\section{Finite groups without exceptional components}

All the results of the previous sections refer to strongly monomial finite groups. In this section we drop this assumption.
Instead we suppose that the Wedderburn decomposition of the rational group algebra does not contain some simple algebras
which
we call exceptional components. An exceptional component of $\Q G$ is an epimorphic image of $\Q G$ which is either a
non-commutative division algebra other than a totally definite quaternion algebra, or a two-by-two matrix algebra over a
division algebra with center either the rationals or a quadratic imaginary extension of $\Q$. Note that our notion of
exceptional is less restrictive than in the references \cite{RS1989,Sehgal1993,JespersLeal1993}. This is possible by using
stronger results on the congruence subgroup problem than those used in the mentioned references (see
Theorem \ref{congruence}).

Let $\O$ be an order in a division ring $D$. For an ideal $Q$ of $\O$ we denote by $E_{n}(Q)$ the subgroup of $\SL_{n}(\O)$
generated by all $Q$-elementary matrices, that is $$E_{n}(Q)=\GEN{I+ qE_{ij} \mid  q\in Q, 1\leq i,j\leq n, i\neq j, E_{ij}
\mbox{ a matrix unit}},$$ where $I$ is the identity matrix in $M_n(\O)$ and $E_{ij}$ denotes the matrix unit having $1$ at
the $(i,j)$-th entry and zeroes elsewhere. For $A,B\in \GL_n(\O)$, we denote $[A,B]=A\inv B\inv AB$.

\begin{lemma}\label{commutators}
 Let $G$ be a finite group such that $\Q G$ has no epimorphic image $M_2(D)$ where the center of $D$ is either $\Q$ or a
quadratic imaginary extension of $\Q$. Denote $\Q G=\bigoplus_e M_{n_e}(D_e)$, where $e$ runs through the primitive central
idempotents of $\Q G$, and let $\O_e$ be an order in $D_e$. Then for every primitive central idempotent $e$ there exists a
non-zero ideal $Q_e$ of $\O_e$ such that $1-e+E_{n_e}(Q_e)\subseteq \U(\Z G)'$.
\end{lemma}
\begin{proof}
Since unit groups of orders are commensurable, for every primitive central idempotent $e$ of $\Q G$, there exists an integer $m$ such that 
$(1-e+A)^m=1-e+A^m\in \U(\Z G)$ for all $A\in \GL_{n_e}(\O_e)$.

If $n_e=1$, there is nothing to prove since $E_1(Q_{e})=\{e\}$ for every ideal $Q_e$ of $\O_e$.

Now consider $n_e=2$. 
In this case, $\mathcal{Z}(D_e)$ is different from $\Q$ and any quadratic imaginary extension of $\Q$. Hence it follows from the Dirichlet Unit
Theorem that $\U(\mathcal{Z}(\O_e))$ is infinite and hence $\O_e$ contains a central unit $r_e$ of infinite order.
It is easy to see that
\begin{eqnarray*}
 \left[1-e+\begin{pmatrix} r_e\inv & 0 \\ 0 & 1 \end{pmatrix}^m, 1-e+\begin{pmatrix} 1 & -s \\ 0 & 1 \end{pmatrix}^m\right] &=& 1+ m(r_e^m-1)s E_{12}, \\
 \left[1-e+\begin{pmatrix} 1 & 0 \\ 0 & r_e\inv \end{pmatrix}^m, 1-e+\begin{pmatrix} 1 & 0 \\ -s & 1 \end{pmatrix}^m \right] &=& 1+ m(r_e^m-1)s E_{21},
\end{eqnarray*} for any $s\in \O_e$. Hence in this case $1-e+E_{2}(m(r_e^m-1)\O_e)\subseteq \U(\Z G)'$. Since $r_e$ is central, $Q_e=m(r_e^m-1)\O_e$ is an ideal of $\O_e$.

Finally assume $n_e\geq 3$. It is easy to verify that
\begin{eqnarray*}
 [1-rE_{ij},1+sE_{jk}]=1+ rsE_{ik}
\end{eqnarray*}
for all $r,s$ in $\O_e$ if $i$, $j$ and $k$ are different. So taking $Q_e=(m\O_e)^2$, it follows that $1-e+E_{n_e}(Q_e)\subseteq \U(\Z G)'$.
\end{proof}

We recall the following celebrated theorem due to Bass, Liehl, Vaser{\v{s}}te{\u\i}n and Venkataramana \cite{Bass1964, Liehl1981,Vaserstein1972, Vaserstein1973, Venkataramana1994}. 

\begin{theorem}\label{congruence}
Let $D$ be a finite dimensional rational division ring and let $\O$ be an order in $D$. If $n\geq 3$ or $n=2$ and $D$ is
different from $\Q$, a quadratic imaginary extension of $\Q$ and a totally definite quaternion algebra with center $\Q$, then
$[\SL_n(\O):E_n(Q)]<\infty$ for any non-zero ideal $Q$ of $\O$.
\end{theorem}

\begin{theorem}\label{TheoremRankU'}
Let $G$ be a finite group such that $\Q G$ has no exceptional components and denote by $\pi$ the natural projection
$\pi:\U(\Z G)\rightarrow \U(\Z G)/{\U(\Z G)}'$. Then
for every torsion-free complement $T$ of $\pm \mathcal{Z}(G)$ in $\mathcal{Z}(\U(\Z G))$, we have
\begin{enumerate}
\item $T\cap \U(\Z G)'=1$ and
\item $\pi(T)$ has finite index in $\U(\Z G)/{\U(\Z G)}'$.
\end{enumerate}
Thus $\pi(T)\simeq T$ and $\U(\Z G)/{\U(\Z G)}'$ has the same rank as $\mathcal{Z}(\U(\Z G))$.
\end{theorem}
\begin{proof}
Let $t\in T\cap {\U(\Z G)}'$. Since $t$ is a central element with reduced norm 1 in all components, $t$ is torsion. Since $T$ is torsion-free, $t=1$.
Hence $T\cap {\U(\Z G)}'$ is trivial and $\pi(T)\simeq T$.

 Denote $\Q G=\bigoplus_e M_{n_e}(D_e)$, where each $D_e$ is a division ring and $e$ runs through the primitive central
idempotents of $\Q G$. Let $\O_e$ be an order in $D_e$.

 By Lemma \ref{commutators}, there exists a non-zero ideal $Q_e$ of $\O_e$ such that $1-e + E_{n_e}(Q_e) \subseteq \U(\Z G)'$ for all primitive 
central idempotents $e$ of $\Q G$. Now define $$S= \prod_e (1-e+E_{n_e}(Q_e))\subseteq {\U(\Z G)}'.$$

It is not hard to prove, using properties of the reduced norm, that $\SL_{n_e}(\O_e)$ together with the center of
$\GL_{n_e}(\O_e)$ generate a subgroup of finite index in $\GL_{n_e}(\O_e)$. Because of the assumption on $\Q G$ and Theorem
\ref{congruence}, it follows that also $E_{n_e}(Q_e)$ together with the center of $\GL_{n_e}(\O_e)$ generate a subgroup of
finite index in $\GL_{n_e}(\O_e)$ for $n_e\geq 2$. This is also true when $n_e=1$, since the only allowed non-commutative
division ring is a totally definite quaternion algebra and hence $\SL_1(\O_e)$ is finite by a result of Kleinert \cite[Lemma
21.3]{Sehgal1993}.

 Since $T$ is a free abelian group of finite index in $\mathcal{Z}(\U(\Z G))$, which is commensurable with $\prod_e \mathcal{Z}(\GL_{n_e}(\O_e))$, 
$T$ is isomorphic to a subgroup of finite index in $\prod_e \mathcal{Z}(\GL_{n_e}(\O_e))$.

 Combining these results, one gets that $T$ is isomorphic to a subgroup of finite index in $$\prod_e
\left(1-e+ \GL_{n_e}(\O_e)\right) / S,$$ since $T\cap S=\{1\}$. Now since $S\subseteq {\U(\Z G)}'$, we have that $\pi(T)$ is
of finite index in $\U(\Z G)/{\U(\Z G)}'$. Since $T$ is a free abelian group, it follows that the rank of $\U(\Z G)/{\U(\Z
G)}'$ equals the rank of $\mathcal{Z}(\U(\Z G))$.
\end{proof}

The next theorem can be proved making use of $K$-theory.

\begin{theorem}\label{commutator}
 Let $G$ be a finite group such that $\Q G$ has no exceptional components. Then the natural images of the Bass units of $\Z G$ generate a subgroup 
of finite index in $\U(\Z G)/{\U(\Z G)}'$.
\end{theorem}
\begin{proof}
Let $B=B(G)$ be the group generated by Bass units in $\Z G$ and consider the natural homomorphism $$B\rightarrow \U(\Z G)/\U(\Z G)'\rightarrow K_1(\Z
G).$$ It is well known that the rank of $K_1(\Z G)$ equals the rank of $\mathcal{Z}(\U(\Z G))$, which equals the rank of $\U(\Z G)/\U(\Z G)'$ by
Theorem \ref{TheoremRankU'}. Let $k$ be this common rank. Since the image of $B$ in $K_1(\Z G)$ has rank $k$ \cite{bass1966}, the image of $B$ in
$\U(\Z G)/\U(\Z G)'$ must have rank $k$ too. Hence this image has finite index in $\U(\Z G)/\U(\Z G)'$. 
\end{proof}

\begin{remark}
 The proof of Theorem \ref{TheoremRankU'} together with the results in \cite{JespersLeal1993} show that the bicyclic units
generate a subgroup of finite index in $\U(\Z G)'$ if $G$ is a finite group such that $G$ has no non-abelian homomorphic
image which is fixed point free and $\Q G$ has no exceptional components.
\end{remark}

\begin{corollary}
  Let $G$ be a finite group. Then the natural images of the Bass units of $\Z G$ generate a subgroup of finite index in $\GL_3(\Z G)/{\GL_3(\Z G)}'$.
\end{corollary}
\begin{proof}
 Note that  $M_3(\Q G)$ has no exceptional components. Therefore we can adapt the proof of Theorem \ref{commutator}.
\end{proof}

The Bass units of $\Z\left[\zeta_{|G|}\right]G$ are all the units of the form: $$\Bass{k}{m}{\zeta_{|G|}^ig}=(1+\zeta_{|G|}^ig+\dots + 
(\zeta_{|G|}^ig)^{k-1})^{m}+(1-k^m)\suma{\zeta_{|G|}^ig},$$
where $g\in G$, $\zeta_{|G|}^ig$ has order $n$, $1\leq j\leq |G|$, $k$ is a positive integer with $k^m \equiv 1 \mod n$ and
$\suma{\zeta_{|G|}^ig}=\frac{1}{n}\sum_{j=0}^{n-1}(\zeta_{|G|}^ig)^j$.

\begin{corollary}
 Let $G$ be a finite group of order different from 1, 2, 3, 4 and 6. Then the natural images of the Bass units of $\Z\left[\zeta_{|G|}\right]G$ 
generate a subgroup of finite index in $\U\left(\Z\left[\zeta_{|G|}\right] G\right)/{\U\left(\Z\left[\zeta_{|G|}\right] G\right)}'$.
\end{corollary}
\begin{proof}
 When $\zeta_{|G|}$ is a primitive $|G|$-th root of unity, $\Q\left(\zeta_{|G|}\right) G$ has no exceptional components
since $\Q\left(\zeta_{|G|}\right)$ is a splitting field of $G$ \cite{JespersLeal1993}. Therefore we can adapt the proof of Theorem \ref{commutator}.
\end{proof}

\renewcommand{\bibname}{References}
\bibliographystyle{amsalpha}

\bibliography{references}

\end{document}